\documentclass[11pt, a4paper, english]{amsart}
\usepackage{amsmath}
\pdfoutput=1
\usepackage{amsthm} 
\usepackage{amssymb} 
\usepackage[dvipsnames]{xcolor}
\usepackage{amscd} 
\usepackage{mathtools}

\usepackage{subcaption}

\usepackage{array} 
\usepackage[cal=boondoxo,scr=euler]{mathalfa}
\usepackage[backref=page,linktocpage]{hyperref} 
\usepackage{cleveref} 
\usepackage{caption} 
\usepackage{graphics,graphicx} 
\usepackage{tikz,tikz-cd} 

\usetikzlibrary{graphs,graphs.standard,calc}

\usetikzlibrary{decorations.pathreplacing,}

\usepackage{enumerate} 

\DeclareMathAlphabet{\mathsf}{OT1}{\sfdefault}{m}{n}

\newcommand{\nocontentsline}[3]{}
\newcommand{\tocless}[2]{\bgroup\let\addcontentsline=\nocontentsline#1{#2}\egroup}

\usepackage[margin=1.25in]{geometry}

\usepackage{verbatim}

\usepackage{scalerel}

\makeatletter
\def\dual#1{\expandafter\dual@aux#1\@nil}
\def\dual@aux#1/#2\@nil{\begin{tabular}{@{}c@{}}#1\\#2\end{tabular}}
\makeatother

\makeatletter
\@namedef{subjclassname@2020}{\textup{2020} Mathematics Subject Classification}
\makeatother

\newcommand{\stirlingtwo}[2]{\biggl\{\genfrac{}{}{0pt}{}{#1}{#2}\biggr\}}
\newcommand{\tstirlingtwo}[2]{\left\{\genfrac{}{}{0pt}{}{#1}{#2}\right\}}

\DeclareMathAlphabet{\amathbb}{U}{bbold}{m}{n}

\hypersetup{
    colorlinks = true,
    linkbordercolor = {white},
    linkcolor = {BrickRed},
    anchorcolor = {black},
    citecolor = {BrickRed},
    filecolor = {cyan},
    menucolor = {BrickRed},
    runcolor = {cyan},
    urlcolor = {BrickRed}
}

\usetikzlibrary{automata}

\newtheoremstyle{teoremas}
{12pt}
{12pt}
{\itshape}
{}
{\bfseries}
{}
{.5em}
{}

\theoremstyle{teoremas}
\newtheorem{theorem}{Theorem}[section]
\newtheorem{corollary}[theorem]{Corollary}

\newtheorem{proposition}[theorem]{Proposition}

\newtheoremstyle{definition}
{12pt}
{12pt}
{}
{}
{\bfseries}
{}
{.5em}
{}

\theoremstyle{definition}
\newtheorem{definition}[theorem]{Definition}

\newtheorem{example}[theorem]{Example}
\newtheorem{remark}[theorem]{Remark}

\crefname{theorem}{theorem}{theorems}
\Crefname{theorem}{Theorem}{Theorems}
\crefname{lemma}{lemma}{lemmas}
\Crefname{lemma}{Lemma}{Lemmas}
\crefname{proposition}{proposition}{propositions}
\Crefname{proposition}{Proposition}{Propositions}

\DeclareMathOperator{\rk}{rk}

\newcommand{\M}{\mathsf{M}}

\newcommand{\U}{\mathsf{U}}

\newcommand{\K}{\mathsf{K}}

\newcommand{\Ind}{\operatorname{Ind}}

\newcommand{\VRep}{\operatorname{VRep}}

\newcommand{\simple}{\operatorname{simp}}
\newcommand{\loopless}{\operatorname{loop}}

\AtBeginDocument{
   \def\MR#1{}
}

\title{Kazhdan--Lusztig polynomials of braid matroids}

\author[L. Ferroni]{Luis Ferroni}

\address{(L. Ferroni)
  Department of Mathematics, KTH Royal Institute of Technology, Stockholm, Sweden
}
\email{ferroni@kth.se}

\author[M. Larson]{Matt Larson}
\address{(M. Larson)
Department of Mathematics, Stanford University, Stanford, USA
}
\email{mwlarson@stanford.edu}


\allowdisplaybreaks

\begin{document}

\begin{abstract}
    We provide a combinatorial interpretation of the Kazhdan--Lusztig polynomial of the matroid arising from the braid arrangement of type $\mathrm{A}_{n-1}$, which gives an interpretation of the intersection cohomology Betti numbers of the reciprocal plane of the braid arrangement. 
    Moreover, we prove an equivariant version of this result. The key combinatorial object is a class of matroids arising from series-parallel networks. As a consequence, we prove a conjecture of Elias, Proudfoot, and Wakefield on the top coefficient of Kazhdan--Lusztig polynomials of braid matroids, and we provide explicit generating functions for their Kazhdan--Lusztig and $Z$-polynomials.
\end{abstract}

\maketitle

\section{Introduction}

\subsection{Overview}

Since their introduction in 1979 \cite{kazhdan-lusztig}, Kazhdan--Lusztig polynomials of Bruhat intervals in Coxeter groups have been extensively studied under various lenses. They play a central role in representation theory, algebraic geometry, and combinatorics. A remarkable feature of the theory of Kazhdan--Lusztig polynomials is that it can be generalized to explain deep phenomena in several areas of mathematics. One far-reaching extension has been proposed by Stanley in \cite{stanley92}, giving rise to what is now known as Kazhdan--Lusztig--Stanley (KLS) polynomials. We refer to \cite{proudfoot} for a detailed survey. Beyond the aforementioned case of Bruhat intervals in Coxeter groups, there are other important families of polynomials with integer coefficients that are encompassed by KLS theory. Notably, the polynomial having as coefficients the $g$-vector of a simplicial polytope is a KLS polynomial. The famous $g$-theorem for simplicial polytopes, established by Billera--Lee \cite{billera-lee} and Stanley \cite{stanley} gives inequalities that the coefficients of the $g$-polynomial must satisfy.

We consider another instance of KLS polynomials which arises from matroids. Matroids can be viewed as a simultaneous abstraction of two pervasive notions in mathematics and computer science:  linear independence and the greedy algorithm. The study of Kazhdan--Lusztig polynomials of matroids was initiated in \cite{elias-proudfoot-wakefield} by Elias, Proudfoot, and Wakefield. This theory played a key role in the resolution of the Dowling--Wilson conjecture by Braden, Huh, Matherne, Proudfoot, and Wang \cite{braden-huh-matherne-proudfoot-wang2}. 

In the three aforementioned cases, i.e., Bruhat intervals in Coxeter groups, simplicial polytopes, and matroids, the KLS polynomials can be defined by relatively simple recursions. The resulting polynomials have non-negative coefficients, but the respective proofs involve the use of heavy machinery from Hodge theory. For example, the non-negativity for $g$-polynomials of simplicial polytopes follows from an application of the Hard Lefschetz theorem \cite{stanley}; in the Coxeter setting, the proof by Elias and Williamson establishes a version of Hodge theory for Soergel bimodules \cite{elias-williamson}; whereas in the matroid case, the central ingredient is a version of Hodge theory for a variant of intersection cohomology for matroids by Braden, Huh, Matherne, Proudfoot, and Wang \cite{braden-huh-matherne-proudfoot-wang2}.

Putting the three proofs into perspective, one may view the non-negativity of these three families of polynomials as a consequence of the fact that they are Hilbert--Poincar\'e polynomials. In other words, the non-negativity of the coefficients follows from the fact that they are the dimensions of suitable vector spaces. However, it remains an outstanding 
and broadly open problem to find \emph{combinatorial interpretations} for their coefficients. 

The most basic examples of matroids arise from cycle matroids of graphs. The case of the complete graph $\mathsf{K}_n$ on $n$ vertices has been central to various key developments in the theory. The cycle matroids of complete graphs are called \emph{braid matroids}. The pursuit of formulas and algorithms for computing the Kazhdan--Lusztig polynomials of braid matroids eventually led to the introduction of other objects, some of which now lie at the core of the \emph{singular} Hodge theory of matroids, developed in \cite{braden-huh-matherne-proudfoot-wang2}. 

The lattice of flats of the braid matroid $\mathsf{K}_n$ is the partition lattice $\Pi_n$, i.e., the poset of partitions of the set $\{1,\ldots,n\}$ ordered by coarsening. It is well known that partition lattices are tightly related to various objects appearing in other areas of mathematics. For example, the strata in $\overline{M}_{0,n+1}$, the moduli spaces of complex projective lines with $n+1$ marked points, are indexed by certain chains in $\Pi_n$. 

The central result of this paper provides a combinatorial interpretation for the coefficients of the Kazhdan--Lusztig polynomials of all braid matroids. Moreover, we will provide an interpretation for a related object, called the $Z$-polynomial. The interpretation we provide is surprisingly concrete: both the coefficients of the Kazhdan--Lusztig polynomials and the $Z$-polynomials of braid matroids count matroids that arise from series-parallel networks. Furthermore, our combinatorial description of the coefficients yields an interpretation of much finer invariants than the Kazhdan--Lusztig and $Z$-polynomials, which take into account the symmetries of the matroid (the \emph{equivariant} Kazhdan--Lusztig and $Z$-polynomials).

This has a number of relevant applications, such as the possibility of writing down  explicit (but complicated) generating functions, which allow one to compute these polynomials explicitly, circumventing their defining recursion. In addition, our main result allows us to settle a conjecture by Elias, Proudfoot, and Wakefield which asserted a mysterious formula for the leading coefficient of the Kazhdan--Lusztig polynomial of $\mathsf{K}_{2n}$. Some other consequences requiring a more technical discussion are also included below.

\subsection{Kazhdan--Lusztig polynomials of matroids} 

As mentioned earlier, the main objects introduced by Elias, Proudfoot, and Wakefield in \cite{elias-proudfoot-wakefield} are the Kazhdan--Lusztig polynomials of matroids. 
To state their precise definition we will follow an alternative approach by Proudfoot, Xu, and Young \cite{PXY}, and Braden and Vysogorets \cite{braden-vysogorets}. Precisely, they show that there is a unique way to associate to every loopless matroid $\M$ a polynomial $P_{\M}(t) \in \mathbb{Z}[t]$ satisfying the following properties:
\begin{enumerate}[\normalfont(i)]
    \item If $\rk(\M)=0$, then $P_{\M}(t) = 1$.
    \item If $\rk(\M)>0$, then $\deg P_{\M}(t) < \frac{1}{2}\rk(\M)$.
    \item For every matroid $\M$, the polynomial
    \[Z_{\M}(t) = \sum_{F \in \mathcal{L}(\M)} t^{\rk(F)} P_{\M/F}(t)\]
    is palindromic\footnote{A polynomial $f(t)$ is \emph{palindromic} if $f(t) = t^{d}f(t^{-1})$, where $d=\deg f$.} of degree $\rk(\M)$. 
\end{enumerate}
Here $\mathcal{L}(\M)$ is the lattice of flats of $\M$. The polynomial $P_{\M}(t)$ is called the \emph{Kazhdan--Lusztig polynomial} of $\M$, and the polynomial $Z_{\M}(t)$ is called the \emph{Z-polynomial} of $\M$. It is not difficult to prove that the three properties above uniquely specify $P_{\M}(t)$ and $Z_{\M}(t)$.

Kazhdan--Lusztig polynomials and $Z$-polynomials of matroids display remarkable properties. For example, the coefficients of the Kazhdan--Lusztig polynomial and $Z$-polynomial of a matroid are non-negative, and the coefficients of the $Z$-polynomial are unimodal \cite[Theorem~1.2]{braden-huh-matherne-proudfoot-wang2}. The lattice of flats of $\M$ is modular if and only if $P_{\M}(t)=1$ \cite[Proposition~2.14]{elias-proudfoot-wakefield}. Kazhdan--Lusztig polynomials and $Z$-polynomials were conjectured in \cite[Conjecture~3.2]{gedeon-proudfoot-young-survey} and \cite[Conjecture~5.1]{PXY} to be real rooted. Much work in the literature is devoted to the study of these polynomials for specific classes of matroids, such as uniform matroids \cite{gao-uniform,gao-xie,qniform}, matroids of corank $2$ \cite{ferroni-schroter}, sparse paving \cite{LNR,ferroni-vecchi}, and paving matroids \cite{ferroni-nasr-vecchi}.

When a matroid $\M$ is realized by a linear subspace $L \subseteq \mathbb{C}^n$, the Kazhdan--Lusztig polynomial and $Z$-polynomial of $\M$ admit algebro-geometric interpretations. The $Z$-polynomial of $\M$ is the Poincar\'{e} polynomial of the intersection cohomology of the closure of $L$ in $(\mathbb{P}^1)^n$, and the Kazhdan--Lusztig polynomial is the Poincar\'{e} polynomial of the stalk of the intersection cohomology at the point $(\infty, \dotsc, \infty)$ of $(\mathbb{P}^1)^n$. Equivalently, the Kazhdan--Lusztig polynomial is the Poincar\'{e} polynomial of the intersection cohomology of the (affine) \emph{reciprocal plane} of $L$: the closure of the image of $L$ under the Cremona transform on $\mathbb{A}^{n}$. This proves the non-negativity of the coefficients when $\M$ is realizable; the non-negativity for arbitrary matroids was proved by Braden, Huh, Matherne, Proudfoot, and Wang \cite{braden-huh-matherne-proudfoot-wang2} by combinatorializing this description of the Kazhdan--Lusztig polynomial and $Z$-polynomial. 

When the action of a subgroup $\Gamma$ of $\mathfrak{S}_n$ preserves $L \subseteq \mathbb{C}^n$, there is an action of $\Gamma$ on the intersection cohomology of the closure of $L$ in $(\mathbb{P}^1)^n$ and on the intersection cohomology of the reciprocal plane of $L$. By taking the $\Gamma$-equivariant Poincar\'{e} polynomial, we may then construct versions of the Kazhdan--Lusztig polynomial and $Z$-polynomial which have coefficients in the ring of virtual representations of $\Gamma$, $\operatorname{VRep}(\Gamma)$. This can be extended to all loopless matroids with an action of a finite group:     There is a unique way of associating to each action $\Gamma \curvearrowright \M$ of a group $\Gamma$ on a loopless matroid\footnote{A group action on a matroid is an action on the ground set which sends flats to flats.} $\M$ a polynomial $P_\M^\Gamma(t) \in \VRep(\Gamma)[t]$ in such a way that
\begin{enumerate}[\normalfont(i)]
    \item If $\rk(\M) = 0$, then $P_\M^{\Gamma}(t) = \operatorname{tr}$ is the trivial representation.
    \item If $\rk(\M) > 0$, then $\deg P_\M^\Gamma(t) < \frac{1}{2}\rk(\M)$.
    \item For every $\M$, the polynomial
    \[Z_\M^\Gamma(t) = \sum_{[F] \in \mathcal{L}(\M)/\Gamma} t^{\rk(F)} \Ind_{\Gamma_F}^\Gamma P_{\M/F}^{\Gamma_F}(t)\]
    is palindromic of degree $\rk(\M)$.
\end{enumerate}
Here, $\Gamma_F$ is the stabilizer of $F$ and $\mathcal{L}(\M)/\Gamma$ denotes the quotient of the lattice of flats of $\M$ by the action of $\Gamma$. The polynomial $P^{\Gamma}_{\M}(t)$ is called the \emph{equivariant Kazhdan--Lusztig polynomial}, and $Z^{\Gamma}_{\M}(t)$ is called the \emph{equivariant $Z$-polynomial}. See \cite{gedeon-proudfoot-young,Equivariant}. 

\subsection{The Kazhdan--Lusztig theory of braid matroids}

As mentioned above, the $n$-th braid matroid $\K_n$ is the graphic matroid associated to the complete graph on $n$ vertices; as such, the symmetric group $\mathfrak{S}_n$ acts on $\K_n$ by permuting the vertices of the complete graph. The $n$-th braid matroid has rank $n-1$. The $n$-th braid matroid is realized by the \emph{braid arrangement} of hyperplanes perpendicular to the roots of type $\mathrm{A}_{n-1}$. 

Kazhdan--Lusztig polynomials of braid matroids have been extensively studied, both equivariantly and non-equivariantly. For other graphic matroids, such as the matroid induced by a cycle on $n$ edges $\U_{n-1,n}$ \cite{proudfoot-wakefield-young}, wheel and fan graphs \cite{LXY}, partial saw graphs \cite{braden-vysogorets}, and thagomizer graphs \cite{Gedeon,xie-zhang}, there are fairly simple formulas for the Kazhdan--Lusztig polynomial. 

The lack of understanding of the Kazhdan--Lusztig polynomials of braid matroids has been the driving force for many significant developments in the theory. For example, the authors of \cite{PXY} state that their ``main motivation'' for introducing the $Z$-polynomials of matroids was to compute the Kazhdan--Lusztig polynomials of braid matroids.

To calculate $P_{\K_{n}}(t)$ and $Z_{\K_{n}}(t)$ using the definitions is in general a heavy computational task. The main result of Karn and Wakefield in \cite{KW} provides an explicit (i.e., non-recursive), but fairly complicated expression for the coefficients of $P_{\K_{n}}(t)$. In practice, the fastest known way to compute both $P_{\K_{n}}(t)$ and $Z_{\K_{n}}(t)$ is provided by the recurrence derived in \cite[Corollary~4.2]{PXY}.

The $\mathfrak{S}_n$-equivariant Kazhdan--Lusztig polynomial of $\K_n$ has also been studied. In \cite{gedeon-proudfoot-young}, the authors computed the linear term of $P^{\mathfrak{S}_n}_{\K_n}(t)$ and gave a functional equation which is satisfied by the generating function of $P^{\mathfrak{S}_n}_{\K_n}(t)$.
In \cite{PY}, the authors studied $P^{\mathfrak{S}_n}_{\K_n}(t)$ using tools from representation stability; this machinery allowed them to give a partial description of the poles of the generating function and bound which irreducible representations of $\mathfrak{S}_{n}$ can appear. See \cite{TOSTESON} for a strengthening of these results. 

A further conjecture posed originally by Elias, Proudfoot, and Wakefield \cite[Section~A]{elias-proudfoot-wakefield} asserts that the leading coefficient of $P_{\K_{2n}}(t)$ counts labelled triangular cacti on $2n-1$ nodes. Prior to the present paper, not only did this conjecture remain elusive; the problem of formulating an analogous statement for the leading coefficient of $P_{\K_{2n-1}}(t)$, or other coefficients, remained open.

Theorem~\ref{thm:main-non-equivariant} below provides a concrete combinatorial interpretation for \emph{all} the coefficients of both $P_{\K_n}(t)$ and $Z_{\K_n}(t)$ for \emph{all} $n$. 

\subsection{Main results} 

A matroid is \emph{quasi series-parallel} if it does not contain $\U_{2,4}$ or $\K_4$ as a minor. Equivalently, a matroid is quasi series-parallel if it is a direct sum of loops, coloops, and series-parallel matroids on a ground set of size at least $2$ (these are defined in Section~\ref{subsec:2.1}). The first of our main theorems can be stated as follows.

\begin{theorem}\label{thm:main-non-equivariant}
    Let $\mathcal{A}(n,r)$ denote the set of all quasi series-parallel matroids on $[n]$ of rank $r$ and, let $\mathcal{S}(n,r)$ denote the set of  simple quasi series-parallel matroids. Then
    \begin{align*} 
        [t^i]\, P_{\K_{n}}(t) &= |\mathcal{S}(n - 1,n-1-i)|, \\
        [t^i]\, Z_{\K_{n}}(t) &= |\mathcal{A}(n - 1,n-1-i)|.
    \end{align*}
\end{theorem}

In fact, we will derive the preceding statement from a stronger result. We provide a description of the equivariant Kazhdan--Lusztig polynomial and the equivariant $Z$-polynomial of all braid matroids with respect to any $\mathfrak{S}_{n-1}$ subgroup of $\mathfrak{S}_n$. 

\begin{theorem}\label{thm:main-equivariant}
    Choose an $\mathfrak{S}_{n-1}$ subgroup of $\mathfrak{S}_n$. Then
    \begin{enumerate}[\normalfont(i)]
        \item The $\mathfrak{S}_{n-1}$-equivariant Kazhdan--Lusztig polynomial of $\K_n$ has $t^i$ coefficient given by the permutation representation of $\mathfrak{S}_{n-1}$ on $\mathcal{S}(n-1, n-1-i)$. 
        \item The $\mathfrak{S}_{n-1}$-equivariant $Z$-polynomial of $\K_n$ has $t^i$ coefficient given by the permutation representation of $\mathfrak{S}_{n-1}$ on $\mathcal{A}(n-1, n-1-i)$. 
    \end{enumerate}
\end{theorem}

An immediate consequence of our main results and  Poincar\'{e} duality and the Hard Lefschetz theorem for intersection cohomology is that the rank-indexed sequence of the numbers of quasi series-parallel matroids on $[n]$ is unimodal. More so, it is a $\gamma$-positive sequence, by \cite[Theorem~1.8]{FMSV}.

We also use Theorem~\ref{thm:main-non-equivariant} to confirm the conjecture in \cite[Section~A]{elias-proudfoot-wakefield} on the leading coefficient of $P_{\K_{2n}}(t)$; see Proposition~\ref{prop:evensize}. In \cite[Remark 6.3]{PY}, Proudfoot and Young note that this conjecture shows the existence of a certain pole of the generating function for the $i$-th coefficient of $P_{\K_{n}}(t)$. 
We also give an expression for the leading term of $P_{\K_{2k+1}}(t)$; see Proposition~\ref{prop:odd}.

Since the study of several objects related to (quasi) series-parallel matroids is a recurring problem in enumerative combinatorics, we are able to present explicit (but complicated) generating functions for Kazhdan--Lusztig and $Z$-polynomials of braid matroids. In particular, this provides an additional tool to study asymptotics of the Betti numbers and the total dimensions of intersection cohomologies of braid matroids, cf. Remark~\ref{rem:moon} below.

\section{Background}

Throughout this paper we shall assume familiarity with the basic properties of matroids. In particular, we mostly follow the notation and terminology of Oxley~\cite{oxley}.

\subsection{Series-parallel matroids}\label{subsec:2.1}

Series-parallel matroids are a prominent family of matroids which pervade graph theory and the theory of electrical networks. Given the variety of sources, coming from both graph theory and enumerative combinatorics, that define similar but different objects under the name ``series-parallel'', we will include a recapitulation of the basic terminology and background of this topic in matroid theory, following \cite[Section~5.4]{oxley}.

Two operations play an important role in the construction of these matroids. Assume that $\M$ and $\M'$ are matroids on $E$ and $E'$ respectively. We say that $\M'$ is a \emph{parallel extension} of $\M$ if there is an element $e\in E'$ which belongs to a circuit of size $2$ in $\M'$ and satisfies that $\M' \setminus e=\M$. Dually, we say that $\M'$ is a \emph{series extension} of $\M$ if there is an element $e\in E'$ which belongs to a cocircuit of size $2$ in $\M'$ and satisfies $\M'/e = \M$. In particular, notice that, if we take $\M = \U_{1,1}$ and $\M'=\U_{2,2}$, then $\M'$ is \emph{not} a series extension of $\M$.

By definition, a \emph{series-parallel matroid} is a matroid that is obtained from a single loop or a single coloop via a finite (possibly empty) sequence of series or parallel extensions\footnote{This follows Oxley's conventions. In particular, all series-parallel matroids on a ground set of size at least $2$ are actually obtained via series and parallel extensions starting from the matroid $\U_{1,2}$.}. We can deduce the following elementary conclusions from the definition of these matroids.
    \begin{itemize}
        \item Series-parallel matroids are connected matroids.
        \item The rank of a series-parallel matroid $\M$ on a ground set of size at least $2$ is equal to one plus the number of series extensions performed in the construction of $\M$ from $\U_{1,2}$. 
    \end{itemize}

A classical paper by Brylawski~\cite{brylawski} establishes the following list of equivalences, reformulated also by Oxley in \cite[Theorem~2.1]{oxley-beta}.

\begin{proposition}[{\cite[Theorem~7.6]{brylawski}}]
    Let $\M$ be a matroid. The following are equivalent.
    \begin{enumerate}[\normalfont (i)]
        \item $\M$ is a series-parallel matroid.
        \item $\M\cong \U_{0,1}$  or $\beta(\M) = 1$.
        \item $\M$ is connected and does not contain any minors isomorphic to $\K_4$ or $\U_{2,4}$.
    \end{enumerate}
\end{proposition}

Here $\beta(\M)$ is the \emph{$\beta$-invariant} of $\M$ introduced by Crapo~\cite{crapo}. Notice that $\beta(\M)=\beta(\M^*)$ whenever $\M\not\cong\U_{0,1},\U_{1,1}$. On the other hand, it follows either by the definition or the above result that $\M$ is series-parallel if and only if $\M^*$ is series-parallel --- even in the case in which $\M$ is either a loop or a coloop.

\begin{table}[htb]
    
    \begin{subtable}[t]{.5\textwidth}
        \caption{Series-parallel matroids\\ on $[n]$ of rank $k$}
        \raggedright
        \begin{tabular}{l l l l l l l l l}\hline
        $k\backslash n=$    & 1   & 2 & 3 & 4 & 5 & 6 & 7  \\ \hline
        0 &  1 & 0 & 0 & 0 & 0 & 0 & 0  \\
        1 &   1 & 1 & 1 & 1 & 1 & 1 & 1  \\
        2 &    & 0 & 1 & 6 & 25 & 90 & 301  \\
        3 &    &  & 0 & 1 & 25 & 290 & 2450  \\
        4 &    &  &  & 0 & 1 & 90 & 2450  \\
        5 &    &  &  &  & 0 & 1 & 301  \\
        6 &    &  &  &  &  & 0 & 1  \\
        7 &    &  &  &  &  &  & 0 \\
        \end{tabular}
    \end{subtable}
    \begin{subtable}[t]{.5\textwidth}
        \raggedleft
        \caption{Quasi series-parallel matroids\\ on $[n]$ of rank $k$}
        \begin{tabular}{l l l l l l l l l l l l l l l}\hline
        $k\backslash n=$   & 1   & 2 & 3 & 4 & 5 & 6 & 7   \\ \hline
        0 &   1 & 1 & 1 & 1 & 1 & 1 & 1  \\
        1 &   1 & 3 & 7 & 15 & 31 & 63 & 127  \\
        2 &    & 1 & 7 & 35 & 155 & 651 & 2667  \\
        3 &    &  & 1 & 15 & 155 & 1365 & 10941  \\
        4 &    &  &  & 1 & 31 & 651 & 10941  \\
        5 &    &  &  &  & 1 & 63 & 2667  \\
        6 &    &  &  &  &  & 1 & 127  \\
        7 &    &  &  &  &  &  & 1  \\
        \end{tabular}
    \end{subtable}
    \caption{Enumeration of series-parallel and quasi series-parallel matroids.}
    \label{table}
\end{table}

\begin{example}\label{ex:one-graph}
    Using the recursive definition of series-parallel matroids, one may check that there are exactly $4$ isomorphism classes of series-parallel matroids on $6$ elements of rank $3$. They are induced by the graphs depicted in Figure~\ref{fig:series-parallel}. Notice that the first of the four matroids is isomorphic to $20$ matroids on $[6]$, the second one is isomorphic to $180$ matroids on $[6]$, the last two are each isomorphic to $45$ matroids on $[6]$. These add up to $20+180+45+45 = 290$ series-parallel matroids with ground set $\{1,2,3,4,5,6\}$ of rank $3$ (see Table~\ref{table}).
    \begin{figure}[ht]
        \centering
    	\begin{tikzpicture}  
    	[scale=0.95,auto=center,every node/.style={circle,scale=0.7, fill=black, inner sep=2.7pt}] 
    	\tikzstyle{edges} = [thick];
    	
    	\node (a1) at (0,0) {};  
    	\node (a2) at (2,0)  {};
    	\node (a3) at (2,2)  {};  
    	\node (a4) at (0,2) {};
    	
    	\draw[edges] (a1) edge[bend left = 20] (a2);
    	\draw[edges] (a1) edge[bend right = 20] (a2);
    	\draw[edges] (a1) edge (a2);
    	\draw[edges] (a2) edge (a3);
    	\draw[edges] (a3) edge (a4);
    	\draw[edges] (a4) edge (a1);
    	\end{tikzpicture}\qquad 
    	\begin{tikzpicture}  
    	[scale=0.95,auto=center,every node/.style={circle,scale=0.7, fill=black, inner sep=2.7pt}] 
    	\tikzstyle{edges} = [thick];
    	
    	\node (a1) at (0,0) {};  
    	\node (a2) at (2,0)  {};
    	\node (a3) at (2,2)  {};  
    	\node (a4) at (0,2) {};
    	
    	\draw[edges] (a1) edge[bend left = 20] (a2);
    	\draw[edges] (a1) edge[bend right = 20] (a2);
    	\draw[edges] (a2) edge (a3);
    	\draw[edges] (a3) edge (a4);
    	\draw[edges] (a4) edge (a1);
    	\draw[edges] (a1) edge (a3);
    	\end{tikzpicture}
    	\qquad
    	\begin{tikzpicture}  
    	[scale=0.95,auto=center,every node/.style={circle,scale=0.7, fill=black, inner sep=2.7pt}] 
    	\tikzstyle{edges} = [thick];
    	
    	\node (a1) at (0,0) {};  
    	\node (a2) at (2,0)  {};
    	\node (a3) at (2,2)  {};  
    	\node (a4) at (0,2) {};
    	
    	\draw[edges] (a1) edge[bend left = 20] (a3);
    	\draw[edges] (a1) edge[bend right = 20] (a3);
    	\draw[edges] (a1) edge (a2);
    	\draw[edges] (a1) edge[bend right = 20,color=white] (a2);
    	\draw[edges] (a2) edge (a3);
    	\draw[edges] (a3) edge (a4);
    	\draw[edges] (a4) edge (a1);
    	\end{tikzpicture}
    	\qquad
    	\begin{tikzpicture}  
    	[scale=0.95,auto=center,every node/.style={circle,scale=0.7, fill=black, inner sep=2.7pt}] 
    	\tikzstyle{edges} = [thick];
    	
    	\node (a1) at (0,0) {};  
    	\node (a2) at (2,0)  {};
    	\node (a3) at (2,2)  {};  
    	\node (a4) at (0,2) {};
    	
    	\draw[edges] (a1) edge[bend left = 20] (a2);
    	\draw[edges] (a1) edge[bend right = 20] (a2);
    	\draw[edges] (a3) edge[bend left = 20] (a2);
    	\draw[edges] (a3) edge[bend right = 20] (a2);
    	\draw[edges] (a3) edge (a4);
    	\draw[edges] (a4) edge (a1);
    	\end{tikzpicture}
    	\caption{Isomorphism classes of series-parallel matroids on $[6]$ of rank $3$.}\label{fig:series-parallel}
    \end{figure}
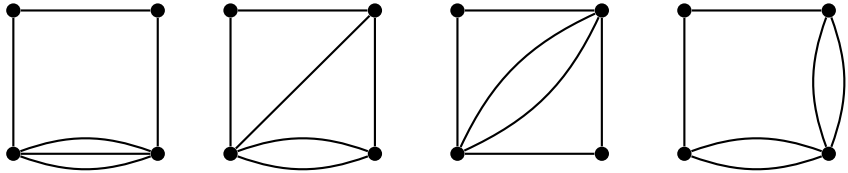
\end{example}

\begin{proposition}
    Let $\mathcal{C}(n,k)$ denote the set of series-parallel matroids on $[n]$ of rank $k$. Then, the bivariate generating function 
    \[ C(x,y) := \sum_{n=1}^{\infty} \sum_{k=0}^n |\mathcal{C}(n,k)| \frac{x^n}{n!} \, y^k ,\]
    is given by
    \begin{equation}\label{eq:C-gen-fun}
        C(x,y) = x(y+1) + y\int \left[ \left( \frac{1}{y} \log(1+xy) + \log(1+x) - x\right)^{\left<-1\right>}\right] dx.
    \end{equation}
    Here the symbol $\left<-1\right>$ stands for the compositional inverse of the function inside the parenthesis with respect to the variable $x$, i.e., treating $y$ as a constant.
\end{proposition}

\begin{proof}
    The case $y=1$ of the above formula reduces to a classical problem studied extensively in the literature; see \cite[Solution of Exercise~5.40]{stanley-ec2} and references mentioned therein. The case of our interest was essentially carried out by Drake in his PhD thesis \cite[Example~1.5.1]{drake-thesis}. By relying on a combinatorial interpretation for Lagrange's inversion theorem, he establishes a generating function for ``series-parallel networks'' on a labelled ground set of edges according to the number of series extensions. Explicitly:
    \begin{align*}
        \left(\frac{1}{\alpha}\log(1+\alpha x) + \frac{1}{\beta}\log(1+\beta x) - x\right)^{\left<-1\right>} &= x + (\alpha+\beta) \frac{x^2}{2!} + (\alpha^2+6\alpha\beta+\beta^2) \frac{x^3}{3!}+\\
        &\qquad  (\alpha^3+25\alpha^2\beta+25\alpha\beta^2+\beta^3) \frac{x^4}{4!} + \cdots
    \end{align*}
    By taking $\alpha=y$ and $\beta=1$ and integrating, we obtain:
    \begin{align*}
        \int \left(\frac{1}{y}\log(1+xy) + \log(1+x) - x\right)^{\left<-1\right>} dx &= \frac{x^2}{2!} + (y+1) \frac{x^3}{3!} + (y^2+6y+1) \frac{x^4}{4!}+\\
        &\qquad  (y^3+25y^2+25y+1) \frac{x^5}{5!} + \cdots
    \end{align*}
    We must make a small correction by multiplying by $y$ and adding $x(y+1)$. The term $x(y+1)$ comes from considering the two matroids with ground set of size $1$. The reason these do not appear in Drake's formula is because, under the definition he uses for ``series-parallel networks'' as parenthesized expressions under an equivalence relation, these two matroids would come from ``empty'' expressions.
\end{proof}

\begin{remark}
    We do not know a way to simplify~\eqref{eq:C-gen-fun}. We point out that the use of formulas for the antiderivative of an inverse function does not help to get rid of the integral symbol. On the other hand,  equally complicated expressions can be deduced by using Lambert's $W$-function.
\end{remark}

\subsection{Quasi series-parallel matroids}
Although the class of series-parallel matroids is closed under taking duals, it is not closed under taking direct sums or taking minors. The best way to resolve this issue is through the class of quasi series-parallel matroids.

\begin{definition}
    Let $\M$ be a matroid. We say that $\M$ is a \emph{quasi series-parallel matroid} if all the connected components of $\M$ are series-parallel matroids.
\end{definition}

Although the name ``quasi series-parallel matroid'' is new, this class of matroids has actually appeared before in the literature\footnote{In fact, Seymour \cite{seymour} calls ``series-parallel matroids'' what we have called ``quasi series-parallel matroids''. We emphasize once again that we are following Oxley's convention.}.

\begin{proposition}
    Let $\M$ be a matroid. The following are equivalent:
    \begin{enumerate}[\normalfont(i)]
        \item $\M$ is a quasi series-parallel matroid.
        \item $\M^*$ is a quasi series-parallel matroid.
        \item $\M$ does not contain any minor isomorphic to $\K_4$ or $\U_{2,4}$.
        \item $\M$ is a binary gammoid.
        \item $\M$ is a regular gammoid.
        \item $\M$ has branch-width smaller than or equal to $2$.
    \end{enumerate}
\end{proposition}

The above result is essentially a restatement of \cite[Theorem~10.4.8]{oxley} and \cite[Proposition~14.2.6]{oxley}; we refer to Oxley's book for the undefined terminology. 

\begin{example}
    There are exactly $6$ isomorphism classes of quasi series-parallel matroids on $4$ elements with rank $2$. They are depicted in Figure~\ref{fig:quasi-series-parallel}. The number of different matroids on $\{1,2,3,4\}$ isomorphic to each of these matroids is $6$, $12$, and $6$ for the three matroids on the top, from left to right, and $4$, $3$, and $4$ respectively for the three on the bottom. This amounts to $6+12+6+4+3+4=35$ quasi series-parallel matroids with ground set $\{1,2,3,4\}$ of rank $2$ (see Table~\ref{table}).
    \begin{figure}[ht]
        \centering
    	\begin{tikzpicture}  
    	[scale=0.85,auto=center,every node/.style={circle,scale=0.7, fill=black, inner sep=2.7pt}] 
    	\tikzstyle{edges} = [thick];
    	
    	\node (a1) at (0,0) {};  
    	\node (a2) at (2,0)  {};
    	\node (a3) at (4,0)  {};
    	
    	\draw[edges] (a1) to[out=135,in=45,looseness=20] (a1);
    	\draw[edges] (a1) edge (a2);
    	\draw[edges] (a2) edge (a3);
    	\draw[edges] (a2) edge[bend right = 30,color=white] (a3);
    	\draw[edges] (a3) to[out=135,in=45,looseness=20] (a3);
    	\end{tikzpicture}\enspace
    	\begin{tikzpicture}  
    	[scale=0.85,auto=center,every node/.style={circle,scale=0.7, fill=black, inner sep=2.7pt}] 
    	\tikzstyle{edges} = [thick];
    	
    	\node (a1) at (0,0) {};  
    	\node (a2) at (2,0)  {};
    	\node (a3) at (4,0)  {};
    	
    	\draw[edges] (a1) to[out=135,in=45,looseness=20] (a1);
    	\draw[edges] (a1) edge (a2);
    	\draw[edges] (a2) edge[bend left = 30] (a3);
    	\draw[edges] (a2) edge[bend right = 30] (a3);
    	\end{tikzpicture}\qquad\enspace
    	\begin{tikzpicture}  
    	[scale=0.85,auto=center,every node/.style={circle,scale=0.7, fill=black, inner sep=2.7pt}] 
    	\tikzstyle{edges} = [thick];
    	
    	\node (a1) at (0,0) {};  
    	\node (a2) at (2,0)  {};
    	\node (a3) at (1,1.73)  {};
    	
    	\draw[edges] (a2) edge[bend left = 20] (a3);
    	\draw[edges] (a2) edge[bend right = 20] (a3);

    	\draw[edges] (a1) edge (a2);
    	\draw[edges] (a3) edge (a1);
    	\end{tikzpicture}
    	
    	\vspace{0.6cm}
    	
    	\qquad
    	\begin{tikzpicture}  
    	[scale=0.85,auto=center,every node/.style={circle,scale=0.7, fill=black, inner sep=2.7pt}] 
    	\tikzstyle{edges} = [thick];
    	
    	\node (a1) at (0,0) {};  
    	\node (a2) at (2,0)  {};
    	\node (a3) at (4,0)  {};
    	
    	\draw[edges] (a1) edge (a2);
    	\draw[edges] (a2) edge (a3);
    	\draw[edges] (a2) edge[bend left = 30] (a3);
    	\draw[edges] (a2) edge[bend right = 30] (a3);
    	\end{tikzpicture}
    	\qquad\qquad
    	\begin{tikzpicture}  
    	[scale=0.85,auto=center,every node/.style={circle,scale=0.7, fill=black, inner sep=2.7pt}] 
    	\tikzstyle{edges} = [thick];
    	
    	\node (a1) at (0,0) {};  
    	\node (a2) at (2,0)  {};
    	\node (a3) at (4,0)  {};
    	
    	\draw[edges] (a2) edge[bend left = 30] (a1);
    	\draw[edges] (a2) edge[bend right = 30] (a1);
    	\draw[edges] (a2) edge[bend left = 30] (a3);
    	\draw[edges] (a2) edge[bend right = 30] (a3);
    	\end{tikzpicture}\enspace
    	\begin{tikzpicture}  
    	[scale=0.85,auto=center,every node/.style={circle,scale=0.7, fill=black, inner sep=2.7pt}] 
    	\tikzstyle{edges} = [thick];
    	
    	\node (a1) at (0,0) {};  
    	\node (a2) at (2,0)  {};
    	\node (a3) at (1,1.73)  {};
    	
        \draw[edges] (a1) to[out=90,in=180,looseness=20] (a1);
    	\draw[edges] (a1) edge (a2);
    	\draw[edges] (a2) edge (a3);
    	\draw[edges] (a3) edge (a1);
    	\end{tikzpicture}
    	
    	\caption{Isomorphism classes of quasi series-parallel matroids on $[4]$ of rank~$2$.}\label{fig:quasi-series-parallel}
    \end{figure}
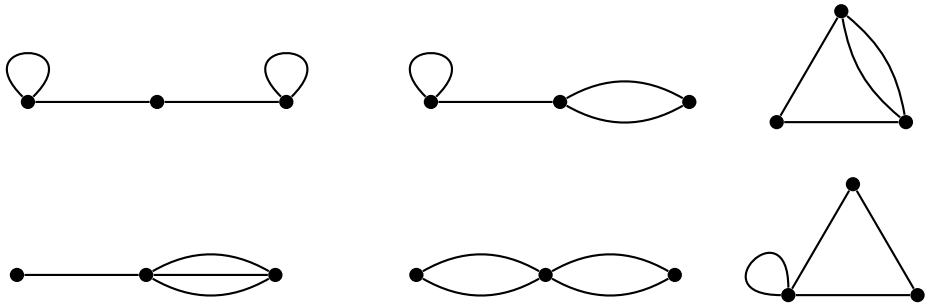
\end{example}

A recurring theme within combinatorics consists of passing from the enumeration of connected labelled structures to the enumeration of all labelled structures (whether connected or not). In particular, the theory of exponential generating functions is useful to this end. Moreover, in the case of our interest, equation~\eqref{eq:C-gen-fun} enumerates all connected quasi series-parallel matroids. If we want to enumerate all quasi series-parallel matroids, by a standard technique (see \cite[Section~5.1]{stanley-ec2} or \cite[p.~148]{flajolet-sedgewick}) it suffices to compose this function with an exponential.

\begin{proposition}\label{prop:gen-fun-A}
    Let $\mathcal{A}(n,k)$ denote the set of all quasi series-parallel matroids on $[n]$ of rank $k$. Then, the bivariate generating function 
    \[ A(x,y) := \sum_{n=1}^{\infty} \sum_{k=0}^n |\mathcal{A}(n,k)| \frac{x^n}{n!} \, y^k ,\]
    is given by
    \begin{equation}\label{eq:A-gen-fun}
        A(x,y) = \exp(C(x,y)),
    \end{equation}
    where $C(x,y)$ is given as in equation~\eqref{eq:C-gen-fun}.
\end{proposition}

\subsection{Simple quasi series-parallel matroids}

The last family of matroids playing a role in the statement of Theorem~\ref{thm:main-non-equivariant} is that of simple quasi series-parallel matroids.

Recall that a matroid $\M$ is \emph{simple} if it loopless and contains no circuits of size $2$. Two (possibly equal)  elements $i, j$ are said to be \emph{parallel} if they are not loops and $\rk_{\M}(\{i, j\}) = 1$. The ground set of $\M$ is partitioned into subsets called \emph{parallel classes}, which are the maximal subsets that consist of elements which are parallel to each other. The \emph{simplification} of $\M$, denoted $\overline{\M}$, is the matroid on the set of parallel classes of $\M$ with rank function 
\[ \rk_{\overline{\M}}(S) = \rk_{\M}\left(\textstyle\bigcup_{T \in S} T\right) \text{ for any S}.\]

\begin{example}
    There are exactly two isomorphism classes of simple quasi series-parallel matroids on $7$ elements of rank $4$. They are depicted in Figure~\ref{fig:simple-quasi-series-parallel}. There are $630$ matroids on $[7]$ isomorphic to the matroid on the left, and there are $105$ for the matroid on the right. These add up to $735$ simple quasi series-parallel matroids on $[7]$ of rank $4$ (see Table~\ref{table2}).
    \begin{figure}[ht]
        \centering
        \begin{tikzpicture}  
    	[scale=1.15,auto=center,every node/.style={circle,scale=0.7, fill=black, inner sep=2.7pt},every edge/.append style = {thick}] 
    	\tikzstyle{edges} = [thick];
        \graph  [empty nodes, clockwise, radius=1em,
        n=9, p=0.3] 
            { subgraph C_n [n=5,clockwise,radius=1.5cm,name=A]};
            
		\draw[edges] (A 4) edge (A 1);
		\draw[edges] (A 4) edge (A 2);
    \end{tikzpicture}\qquad\qquad
    	\begin{tikzpicture}  
    	[scale=0.95,auto=center,every node/.style={circle,scale=0.7, fill=black, inner sep=2.7pt}] 
    	\tikzstyle{edges} = [thick];
    	
    	\node (a1) at (0,0) {};  
    	\node (a2) at (2,0)  {};
    	\node (a3) at (2,2)  {};  
    	\node (a4) at (0,2) {};
    	\node (a5) at (1,1) {};
    	
    	\draw[edges] (a1) edge (a2);
    	\draw[edges] (a2) edge (a3);
    	\draw[edges] (a1) edge (a5);
    	\draw[edges] (a5) edge (a3);
    	\draw[edges] (a1) edge[bend left = 30] (a3);
    	\draw[edges] (a3) edge (a4);
    	\draw[edges] (a4) edge (a1);
    	\end{tikzpicture}
    	\caption{Isomorphism classes of simple quasi series-parallel matroids on $[7]$ of rank $4$.}\label{fig:simple-quasi-series-parallel}
    \end{figure}
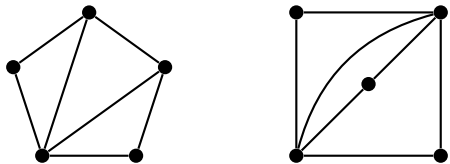
\end{example}

\begin{proposition}\label{prop:rankbound}
If $n > 1$, then a simple quasi series-parallel matroid on $[n-1]$ has rank greater than $\frac{n-1}{2}$. 
\end{proposition}

\begin{proof}
We may reduce to the case of simple (connected) series-parallel matroids on nonempty ground sets. Then the result follows from the observation that in order to obtain a simple matroid from a sequence of series and parallel extensions of $\U_{1,1}$, we must have done at least as many series extensions as parallel extensions. 
\end{proof}

\begin{table}[htb]
    \begin{subtable}[t]{.5\textwidth}
        \caption{Simple quasi series-parallel matroids on $[n]$ of rank $k$}
        \raggedright
        \begin{tabular}{l l l l l l l l l}\hline
             $k\backslash n=$    & 1   & 2 & 3 & 4 & 5 & 6 & 7 & 8 \\ \hline
            0 & 0 & 0 & 0 & 0 & 0 & 0 & 0 & 0\\
            1 & 1 & 0 & 0 & 0 & 0 & 0 & 0 & 0\\
            2 &  & 1 & 1 & 0 & 0 & 0 & 0 & 0\\
            3 &  &  & 1 & 5 & 15 & 0 & 0 & 0\\
            4 &  &  &  & 1 & 16 & 175 & 735 & 0\\
            5 &  &  &  &  & 1 & 42 & 1225 & 16065\\
            6 &  &  &  &  &  & 1 & 99 & 6769\\
            7 &  &  &  &  &  &  & 1 & 219\\
            8 &  &  &  &  &  &  &  & 1
        \end{tabular}
    \end{subtable}
    \caption{Enumeration of simple quasi series-parallel matroids.}\label{table2}
\end{table}

We now study simple quasi series-parallel matroids of the smallest possible rank. We begin with the case of ground sets of odd size. By Proposition~\ref{prop:rankbound}, a simple quasi series-parallel matroid of rank $k$ on $[2k-1]$ is connected, and thus is a series-parallel matroid. 
A \emph{triangular cactus} is a connected graph where every edge is contained in a unique cycle, and that cycle is a triangle.

\begin{proposition}\label{prop:evensize}
There is a bijection between simple series-parallel matroids of rank $k$ on $[2k-1]$ and triangular cacti with vertex set $[2k-1]$. 
\end{proposition}

Before giving the bijection, we observe that both triangular cacti and rank $k$ simple series-parallel matroids on $[2k-1]$ are built inductively. The triangles in a triangular cactus are arranged in a tree-like fashion, so every triangular cactus with at least three vertices contains a triangle which has at least two vertices of degree $2$. When we delete those two vertices, we get a triangular cactus on a smaller ground set. Every rank $k$ simple series-parallel matroid on $[2k-1]$, with $k \ge 2$, is a series extension of a series-parallel matroid of rank $k-1$ on ground set of size $2k-2$. By Proposition~\ref{prop:rankbound}, such a matroid cannot be simple, and so it is a parallel extension of a rank $k-1$ simple series-parallel matroid on a ground set of size $2k-3$.

\begin{proof}
To a simple series-parallel matroid $\M$ of rank $k$ on $[2k-1]$, we associate the graph $T(\M)$ with vertex set $[2k-1]$ and edges $(a, b)$ if $\{a, b\}$ is contained in a $3$-element circuit of $\M$. Note that if we do a parallel extension and then a series extension at $i \in [2k-1]$ to $\M$ to obtain $\widetilde{\M}$, we add a triangle to $T(\M)$ with vertices $i, 2k$, and $2k+1$. In particular, $T(\widetilde{\M})$ is a triangular cactus by induction. 

To a triangular cactus $G$ with vertex set $[2k-1]$, we build a graph whose edge set is $[2k-1]$ by adding a triangle with edges $\{a, b, c\}$ for each triangle with vertices $\{a, b, c\}$ in $G$. 
That there is such a graph follows from the inductive structure of triangular cacti. 
We obtain a matroid $S(G)$ by taking the graphic matroid of this graph. Note that if $G$ is obtained from a triangular cactus $G'$ on a vertex set of size $2k-3$ by adding a triangle, then $S(G)$ is obtained from $S(G')$ by doing a parallel extension and then a series extension. We see by induction that $S(G)$ is independent of the choice of graph and is a rank $k$ simple series-parallel matroid. Similarly, it follows from induction that $T(S(G)) = G$. 
To check that $S(T(\M)) = \M$, we again use that both constructions are compatible with the inductive structure of series-parallel matroids and triangular cacti. 
\end{proof}

    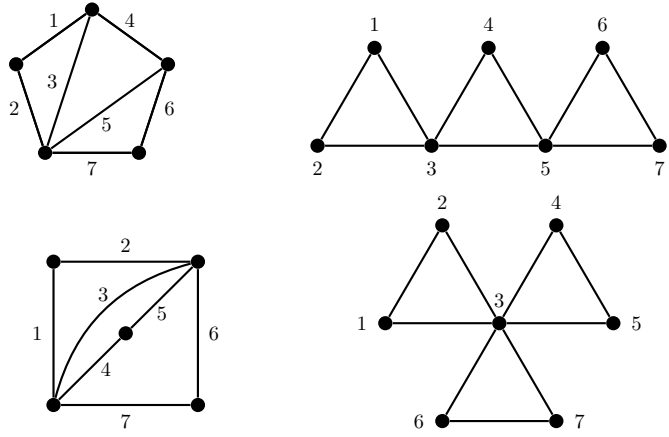
\begin{figure}[ht]
        \centering
        \begin{tikzpicture}  
        [scale=1.15,auto=center,every node/.style={circle,scale=0.7, fill=black, inner sep=2.7pt},every edge/.append style = {thick}]
        \tikzstyle{edges} = [thick];
        \graph  [empty nodes, clockwise, radius=1em,
        n=9, p=0.3] 
            { subgraph C_n [n=5,clockwise,radius=1.5cm,name=A]};
            
        \draw[edges] (A 1) edge node[opacity=0,text opacity=1,above] {$4$} (A 2);
        \draw[edges] (A 2) edge node[opacity=0,text opacity=1,right] {$6$} (A 3);
        \draw[edges] (A 3) edge node[opacity=0,text opacity=1,below] {$7$} (A 4);
        \draw[edges] (A 4) edge node[opacity=0,text opacity=1,left] {$2$} (A 5);
        \draw[edges] (A 5) edge node[opacity=0,text opacity=1,above] {$1$} (A 1);
        \draw[edges] (A 4) edge node[opacity=0,text opacity=1,left] {$3$} (A 1);
        \draw[edges] (A 4) edge node[opacity=0,text opacity=1,below] {$5$} (A 2);

    \end{tikzpicture}\qquad\qquad
            \begin{tikzpicture}  
        [scale=0.75,auto=center,every node/.style={circle,scale=0.7, fill=black, inner sep=2.7pt},every edge/.append style = {thick}] 
        \tikzstyle{edges} = [thick];
        \node[label=below:$2$] (a1) at (0,0) {};
        \node[label=below:$3$]  (a2) at (2,0)  {};
        \node[label=above:$1$] (a3) at (1,1.73)  {};  
        \node[label=below:$5$] (a4) at (4,0) {};
        \node[label=above:$4$] (a5) at (3, 1.73) {};
        \node[label=below:$7$] (a6) at (6, 0) {};
        \node[label=above:$6$] (a7) at (5, 1.73) {};
        \draw[edges] (a1) edge (a2);
        
        \draw[edges] (a2) edge (a3);
        \draw[edges] (a1) edge (a3);
        \draw[edges] (a2) edge (a4);
        \draw[edges] (a2) edge (a5);
        \draw[edges] (a4) edge (a5);
        \draw[edges] (a4) edge (a6);
        \draw[edges] (a4) edge (a7);
        \draw[edges] (a6) edge (a7);
    \end{tikzpicture}

    \begin{tikzpicture}  
    	[scale=0.95,auto=center,every node/.style={circle,scale=0.7, fill=black, inner sep=2.7pt}] 
    	\tikzstyle{edges} = [thick];
    	
    	\node (a1) at (0,0) {};  
    	\node (a2) at (2,0)  {};
    	\node (a3) at (2,2)  {};  
    	\node (a4) at (0,2) {};
    	\node (a5) at (1,1) {};
    	
    	\draw[edges] (a1) edge node[opacity=0,text opacity=1,below] {$7$} (a2);
    	\draw[edges] (a2) edge node[opacity=0,text opacity=1,right] {$6$} (a3);
    	\draw[edges] (a1) edge node[opacity=0,text opacity=1,right] {$4$} (a5);
    	\draw[edges] (a5) edge node[opacity=0,text opacity=1,below] {$5$}  (a3);
    	\draw[edges] (a1) edge[bend left = 30] node[opacity=0,text opacity=1,above] {$3$} (a3);
    	\draw[edges] (a3) edge node[opacity=0,text opacity=1,above] {$2$} (a4);
    	\draw[edges] (a4) edge node[opacity=0,text opacity=1,left] {$1$} (a1);
    	\end{tikzpicture}\qquad\qquad
    	\begin{tikzpicture}  
    	[scale=0.75,auto=center,every node/.style={circle,scale=0.7, fill=black, inner sep=2.7pt}] 
    	\tikzstyle{edges} = [thick];
    	
    	\node[label=above:$3$] (a0) at (0,0) {};
    	\node[label=right:$5$] (a1) at (2,0) {};  
    	\node[label=above:$4$] (a2) at (1,1.73)  {};
    	\node[label=above:$2$] (a3) at (-1,1.73)  {};
    	\node[label=left:$1$] (a4) at (-2,0)  {};
    	\node[label=left:$6$] (a5) at (-1,-1.73)  {};
    	\node[label=right:$7$] (a6) at (1,-1.73)  {};
    	
    	\draw[edges] (a1) -- (a2);
    	\draw[edges] (a3) -- (a4);
    	\draw[edges] (a5) -- (a6);
    	\draw[edges] (a0) -- (a2);
    	\draw[edges] (a0) -- (a3);
    	\draw[edges] (a0) -- (a4);
    	\draw[edges] (a0) -- (a5);
    	\draw[edges] (a0) -- (a6);
    	\draw[edges] (a0) -- (a1);
    	\end{tikzpicture}
        \caption{Two simple series-parallel matroids of rank $4$ on $[7]$ and their corresponding triangular cacti.}\label{fig:cacti}
\end{figure}

In \cite{Bahrani}, the authors show that the number of triangular cacti with vertex set $[2k-1]$ is $(2k-3)!! (2k-1)^{k-2}$. We therefore have the following result, which together with Theorem~\ref{thm:main-non-equivariant} proves a conjecture stated in \cite[Section A]{elias-proudfoot-wakefield}.

\begin{corollary}
    The number of simple quasi series-parallel matroids of rank $k$ on $[2k-1]$ is equal to $(2k-3)!! (2k-1)^{k-2}$.
\end{corollary}

Let $E_k$ be the number of simple series-parallel matroids of rank $k + 1$ on $[2k]$. Then 
$$(E_0, E_1, \dotsc) = (0, 1, 75, 9345, 1865745, 554479695, 231052877055, 128938132548225, \dotsc).$$
We do not know a simple expression for $E_k$; the $E_k$ can have large prime factors. After we prove Theorem~\ref{thm:main-non-equivariant}, the following result will give an expression for the leading term of $P_{\K_{2k+1}}(t)$ in terms of $E_k$. 

\begin{proposition}\label{prop:odd}
The number of simple quasi series-parallel matroids of rank $k + 1$ on $[2k]$ is equal to 
$$E_k + \frac{1}{2} \sum_{a=0}^{k-1} \binom{2k}{2a+1} (2a -1)!! (2k - 2a - 3)!! (2a + 1)^{a-1}  (2k - 2a - 1)^{k - a - 2}.$$
\end{proposition}

\begin{proof}
    By Proposition~\ref{prop:rankbound}, a simple quasi series-parallel matroid of rank $k+1$ on a ground set of size $2k$ can have at most two connected components. The number with one connected component is exactly $E_k$. If there are two connected components, then Proposition~\ref{prop:rankbound} implies that one connected component must have size $2a +1$ and rank $a + 1$ for some $a$, and the other component must have size $2(k-a) - 1$ and rank $k -a$. Proposition~\ref{prop:evensize} then implies the result. 
\end{proof}

Let $\mathcal{F}$ be a family of isomorphism classes of matroids. Denote $\mathcal{F}_{\simple}$ and $\mathcal{F}_{\loopless}$ the subclasses of all of simple and of all loopless matroids in $\mathcal{F}$, respectively. Let us write $|\mathcal{F}(n,k)|$ for the number of rank $k$ matroids on $[n]$ whose isomorphism class lies in $\mathcal{F}$, and similarly for $\mathcal{F}_{\simple}$ and $\mathcal{F}_{\loopless}$. 

Denote by $[\M]$ the isomorphism class of the matroid $\M$. Let us assume that $[\U_{0,0}]\in \mathcal{F}$ and that $\mathcal{F}$ has the property that $[\M]$ lies in $\mathcal{F}$ if and only if the class of the simplification $[\overline{\M}]$ lies in $\mathcal{F}$. This mild assumption allows us to establish the following elementary relationships:
    \begin{align*}
        |\mathcal{F}(n,k)| &= \sum_{i=k}^n \binom{n}{i} |\mathcal{F}_{\loopless}(i,k)|,\\
        |\mathcal{F}_{\loopless}(n,k)| &= \sum_{i=k}^n \stirlingtwo{n}{i} |\mathcal{F}_{\simple}(i,k)|,
    \end{align*}
where $\tstirlingtwo{n}{i}$ denotes the Stirling number of the second kind. Notice that the family $\mathcal{A}$ of all isomorphism classes of quasi series-parallel matroids satisfies this condition.

\begin{proposition}\label{prop:gen-fun-S}
    Let $\mathcal{S}(n,k)$ denote the set of all simple quasi series-parallel matroids on $[n]$ of rank $k$. Then, the bivariate generating function 
    \[ S(x,y) := \sum_{n=1}^{\infty} \sum_{k=0}^n |\mathcal{S}(n,k)| \frac{x^n}{n!} \, y^k ,\]
    is given by
    \begin{equation}\label{eq:S-gen-fun}
        S(x,y) = \frac{1}{x+1} A(\log(x+1),y) - 1,
    \end{equation}
    where $A(x,y)$ is given as in equation~\eqref{eq:A-gen-fun}.
\end{proposition}

\begin{proof}
    This is a formal consequence of the fact that 
    \[ |\mathcal{A}(n,k)| = \sum_{i=k}^n \binom{n}{i}\sum_{j=k}^i \stirlingtwo{i}{j} |\mathcal{S}(j,k)|.\]
    It is possible to translate this expression in terms of generating functions. More precisely, one obtains:
    \[ A(x,y) = e^{-x} \, S(e^x-1,y).\]
    Using the change of variable $u=\log(x+1)$ yields the result of the statement.
\end{proof}

\begin{remark}\label{rem:moon}
    There is a large literature on enumerating various objects which are closely related to (quasi) series-parallel matroids. In the work of Moon \cite{Moon}, an asymptotic estimate of the general term of the single variable series $C(x,1)$, defined in equation~\eqref{eq:C-gen-fun}, was obtained. In particular, one may apply standard techniques to obtain asymptotic estimates for the coefficients of the series $A(x,1)$ and $S(x,1)$. Notice that our Theorem~\ref{thm:main-non-equivariant} implies that the coefficients of these series are described by the $Z$-polynomial and the Kazhdan--Lusztig polynomial, both evaluated at $t=1$. In particular, that provides asymptotics for the total dimensions of the intersection cohomology of $\mathsf{K}_n$ and its stalk at the empty flat.
\end{remark}

\section{Proof of the main results}

The flats of $\K_n$ are in bijection with partitions of $[n]$: a partition $Q$ given by $[n] = S_1 \sqcup \cdots \sqcup S_k$ is identified with the flat $F_Q$ corresponding to the subgraph that consists of all edges where both vertices lie in $S_i$ for some $i$. The simplification of the contraction $\K_n/F_Q$ is the braid matroid on $[k]$, the set of parts in $Q$. The stabilizer in $\mathfrak{S}_n$ of $F_Q$ is the set of permutations which preserve $Q$.

\begin{proposition}\label{prop:simplify}
    Let $\Gamma$ be a group acting on a loopless matroid. Then the subgroup $N$ of $\Gamma$ fixing all of the parallel classes is normal, and $P^{\Gamma}_{\M}(t)$ is the pullback from $\operatorname{VRep}(\Gamma/N)[t]$ of $P^{\Gamma/N}_{\overline{\M}}(t)$. 
\end{proposition}

\begin{proof}
    Any automorphism of a matroid sends parallel classes to parallel classes, so there is a map from $\Gamma$ to the symmetric group on the set of parallel classes whose kernel is $N$. 
    
    The second part follows from induction on the size of the ground set and the fact that simplifying a matroid does not change its lattice of flats. 
\end{proof}

We now prove Theorem~\ref{thm:main-equivariant} from which Theorem~\ref{thm:main-non-equivariant} follows. Our strategy is to use the following observation (see, e.g., \cite[Corollary~A.5]{braden-huh-matherne-proudfoot-wang2}): Let $\Gamma$ be a finite group acting on a matroid $\M$, and let $P^{\Gamma}(t) \in \operatorname{VRep}(\Gamma)$ be a polynomial of degree less than $\frac{1}{2}\rk(\M)$. Suppose that
$$Z^{\Gamma}(t) := P^{\Gamma}(t) + \sum_{\varnothing \not= [F] \in \mathcal{L}(\M)/\Gamma}  t^{\rk(F)} \operatorname{Ind}_{\Gamma_F}^{\Gamma} P_{\M/F}^{\Gamma_F}(t)$$
is palindromic. Then $Z^{\Gamma}(t) = Z^{\Gamma}_{\M}(t)$ and $P^{\Gamma}(t) = P^{\Gamma}_{\M}(t)$. 

\begin{proof}[Proof of Theorem~\ref{thm:main-equivariant}]
We induct on $n$. Let $\Gamma = \mathfrak{S}_{n-1}$ be the subgroup of $\mathfrak{S}_n$ that fixes $n$. Let $P^{\Gamma} \in \operatorname{VRep}(\Gamma)[t]$ be the generating function of the permutation representation of simple quasi series-parallel matroids of rank $n - 1 - i$ on $[n-1]$, and let $Z^{\Gamma}$ be the generating function of the permutation representation of quasi series-parallel matroids of rank $n - 1 - i$ on $[n-1]$. By Proposition~\ref{prop:rankbound}, the degree of $P^{\Gamma}$ is less than $\frac{n-1}{2}$. Because the dual of a quasi series-parallel matroid is quasi series-parallel, $Z^{\Gamma}$ is palindromic. Therefore it suffices to show that
\begin{equation}\label{eq:key}
Z^{\Gamma}(t) = P^{\Gamma}(t) + \sum_{\varnothing \not= [F] \in \mathcal{L}(\K_n)/\mathfrak{S}_{n-1}}  t^{\rk(F)} \operatorname{Ind}_{\Gamma_F}^{\Gamma} P_{\K_n/F}^{\Gamma_F}(t).
\end{equation}

Let $Q = S_1 \sqcup \dotsb \sqcup S_k$ be a partition of $[n]$ with $n \in S_i$, and let $\M$ be a simple quasi series-parallel matroid on $[k] \setminus i$ of rank $k - 1 - j$. The quotient of $\Gamma_{F_Q}$ by the normal subgroup which fixes each parallel class of $\K_n/F_Q$ is a subgroup of $\mathfrak{S}_{[k] \setminus i}$. Let $V_{\M}$ be the pullback of the permutation representation on the set of matroids on $[k] \setminus i$ isomorphic to $\M$ to $\Gamma_{F_Q}$ from $\mathfrak{S}_{[k] \setminus i}$.
By induction and Proposition~\ref{prop:simplify}, $V_{\M}$ is a summand of $[t^j]P_{\K_n/F_Q}^{\Gamma_{F_Q}}(t)$. The stabilizer of $\M$ is the inverse image in $\Gamma_{F_Q}$ of the automorphism group of $\M$. 

Let $\widetilde{\M}$ denote the matroid on $[n-1]$ which has $S_i \setminus n$ as a set of loops, each $S_\ell$ is a parallel class ($\ell \not= i$), and simplifies to $\M$. 
The automorphism group of $\widetilde{\M}$ is the inverse image in $\Gamma_{F_Q}$ of the automorphism group of $\M$, so the induction of $V_{\M}$ to $\mathfrak{S}_{n-1}$ can be identified with the permutation representation on the set of matroids on $[n-1]$ which are isomorphic to $\widetilde{\M}$. Note that $\rk(F_Q) + k - 1 - j = n - 1 - j$, so the permutation representation on the set of matroids on $[n-1]$ isomorphic to $\widetilde{\M}$ appears as a summand in degree $n - 1 - \operatorname{rk}(\widetilde{\M})$. Furthermore, every quasi series-parallel matroid is either simple or simplifies to a simple quasi series-parallel matroid on a smaller ground set, which proves \eqref{eq:key}.
\end{proof}

As a consequence of the main result, we can use the generating functions described in Propositions~\ref{prop:gen-fun-A} and \ref{prop:gen-fun-S} to compute $Z$-polynomials and Kazhdan--Lusztig polynomials of braid matroids.

\begin{corollary}
    The bivariate generating functions $A(x,y)$ and $S(x,y)$ defined by equations \eqref{eq:A-gen-fun} and \eqref{eq:S-gen-fun} satisfy
    \begin{align*}
        {n!}[x^n]\, A(x,y) &= Z_{\mathsf{K}_{n+1}}(y),\\
        {n!}[x^n]\, S(x,y) &= y^{n} P_{\mathsf{K}_{n+1}}(1/y).
    \end{align*}
\end{corollary}

A very short piece of SAGE code implementing these two generating functions is included on the website of the first author\footnote{Accessible at \url{https://sites.google.com/view/ferroniluis/home/research} or by clicking \href{https://drive.google.com/uc?export=view&id=17rbTR30bF9EL1Q9-t0MDJyMDy9JIbU_O}{here}.}.

\section*{Acknowledgments}

The authors want to thank the OEIS \cite{sloane} for invaluable help throughout this project. We thank the referees for a thorough reading and helpful comments. The first author is supported by the Swedish Research Council, Grant 2018-03968. The second author is supported by an NDSEG fellowship.

\bibliographystyle{amsalpha}
\bibliography{bibliography}

\end{document}